\documentclass[12pt]{amsart}
\usepackage{amsmath,amssymb,amsfonts,amsthm,amsopn}
\usepackage{latexsym,graphicx}
\usepackage{pb-diagram}
  \def\cL{\mathcal{L}}

\newcommand{\tfa}{time-frequency analysis}

\newcommand{\stft}{short-time Fourier transform}

\newcommand{\tf}{time-frequency}

\newcommand{\fif}{if and only if}
\newcommand{\tfs}{time-frequency shift}

\newcommand{\psdo}{pseudodifferential operator}

\newtheorem{theorem}{Theorem}[section]
\newtheorem{lemma}[theorem]{Lemma}

\newtheorem{proposition}[theorem]{Proposition}
\newtheorem{definition}[theorem]{Definition}

\newtheorem{remark}[theorem]{Remark}

\newcommand{\beqa}{\begin{eqnarray*}}
\newcommand{\eeqa}{\end{eqnarray*}}

\newcommand{\field}[1]{\mathbb{#1}}
\newcommand{\bR}{\field{R}}        %  real numbers
\newcommand{\bN}{\field{N}}        %  natural numbers
\newcommand{\bZ}{\field{Z}}        %  whole numbers
\newcommand{\fiola}{FIO(\Xi)}
\def\Fu{\mathfrak{F}}
\def\la{\lambda}

\def\cF{\mathcal{F}}              % Calligraphic Letters
\def\cS{\mathcal{S}}

\def\cM{\mathcal{M}}

\def\cC{\mathcal{C}}

\def\a{\aleph}

\def\rd{\bR^d}

\def\rdd{{\bR^{2d}}}

\def\lrd{L^2(\rd)}

\def\intrd{\int_{\rd}}
\def\intrdd{\int_{\rdd}}

\def\R{\right)}

\def\<{\left<}
\def\>{\right>}

\def\mv1{M_v^1}

\def\phas{(x,\o )}
\def\mn{(m,n)}
\def\mn'{(m',n')}

\def\o{\eta}
\def\a{\alpha}

\def\R{\mathbb{R}}
\def\Ren{\mathbb{R}^d}
\def\Renn{\mathbb{R}^{2d}}

\def\Sn2{S_{2}(L^{2}(\Ren))}
\def\S1{S_{1}(L^{2}(\Ren))}
\def\sig00{\sigma_{0,0}}
\def\la{\langle}
\def\ra{\rangle}

\begin{document}
\begin{abstract} We consider the Schr\"odinger equation
\begin{equation*}
 i \displaystyle\frac{\partial
u}{\partial t} +Hu=0,\quad H=a(x,D),
\end{equation*}
where the Hamiltonian $a(z)$, $z=\phas$, is assumed real-valued and smooth, with bounded derivatives $|\partial^\a a(z)|\leq C_\a$, for every $|\a|\geq 2$, $z\in\rdd$. For such equation results are known concerning well-posedness of the Cauchy problem for initial data in $\lrd$ and local representation of the propagator $e^{it H}$ by means of Fourier integral operators.

In the present paper we give a global expression for $e^{itH}$ in terms of Gabor analysis and we deduce boundedness in modulation spaces. Moreover, by using \tf\, techniques, we obtain a result of propagation of micro-singularities for $e^{itH}$. 
\end{abstract}

\title[Gabor analysis for Schr\"odinger equations]{Gabor analysis for Schr\"odinger equations and propagation of singularities}

\author{Elena Cordero}
\address{Universit\`a di Torino, Dipartimento di Matematica, via Carlo Alberto 10, 10123 Torino, Italy}
\email{elena.cordero@unito.it}
\author{Fabio Nicola}
\address{Dipartimento di Scienze Matematiche,
Politecnico di Torino, corso Duca degli Abruzzi 24, 10129 Torino,
Italy}
\email{fabio.nicola@polito.it}
\author{Luigi Rodino}
\address{Universit\`a di Torino, Dipartimento di Matematica, via Carlo Alberto 10, 10123 Torino, Italy}
\email{luigi.rodino@unito.it}


%\date{}
\maketitle
\section{Introduction}

Time-frequency Analysis, also named Gabor Analysis cf. \cite{Daubechies90,book}, has found important applications in Signal Processing and related problems in Numerical Analysis, see for example \cite{cordero-feichtinger-luef,str06} and references therein. More recently, \tf\, methods have been applied to the study of the partial differential equations, in particular constant coefficient wave, Klein-Gordon, parabolic and Schr\"odinger equations \cite{bertinoro2,bertinoro3,bertinoro12,bertinoro17,kki1,kki2,kki3,MNRTT,bertinoro58,bertinoro58bis}, let us also refer to the survey \cite{ruz}
 and the monograph \cite{baoxiang}. The analysis of variable coefficient Schr\"odinger type equations was carried out in \cite{ACHA2015, locNC09, Wiener, fio1, fio3, bertinoro,CNRgabor,CNEdcds,tataru}, see also \cite{CNRanalitico1,CNRanalitico2} in the analytic category.

 In the present paper we address to the Schr\"odinger equation
 \begin{equation}\label{C1intro}
\begin{cases} i \displaystyle\frac{\partial
u}{\partial t} +a(x,D)u=0\\
u(0,x)=u_0(x).
\end{cases}
\end{equation}
 where the symbol $a(z)$, $z=(x,\xi)$, is real-valued and smooth, satisfying
 \begin{equation}\label{symbdec}
   |\partial^\a a(z)|\leq C_\a, \quad |\a|\geq 2, z\in\rdd.
 \end{equation}
In \eqref{C1intro} we understand
\begin{equation*}
a(x,D) f(x)=\intrd e^{2\pi i \la x,\xi\ra }a(x,\xi)\hat{f}(\xi)\,d\xi,
\end{equation*}
with
$$
\hat{f}(\xi)= \intrd e^{-2\pi i \la x,\xi\ra } f(x)\,dx.
$$
As basic example one can consider the case when $a(z)$ is a quadratic form in $z$, including the free particle operator $a(x,D)=-\Delta$ and the harmonic oscillator $a(x,D)=-\Delta+|x|^2$.

Under the assumption  \eqref{symbdec}, it is easy to show by energy methods that the Cauchy problem \eqref{C1intro} is well-posed in $\lrd$ and $\cS(\rd)$, see for example \cite{tataru} and we may consider the propagator
\begin{equation}\label{T1.5eq1}
e^{i t H}: \cS(\rd)\to \cS (\rd),\quad H=a(x,D),
\end{equation}
mapping the initial datum $u_0$ to the solution $u(t,x)$ at time $t\in\bR$.

Going further, one would like to obtain an explicit expression for $e^{it H}$, from which one may obtain precise estimates for the solutions, and numerical analysis. To this end, by using Fourier methods, one expects for small values of $t$ a representation as type $I$ Fourier integral operator
\begin{equation}\label{FIO1}
(e^{it H}u_0)(t,x)
=\intrd e^{2\pi i \Phi(t,x,\o)}b(t,x,\o) \hat{u}_0(\o)\,d\o.
\end{equation}
In fact, under our assumption \eqref{symbdec}, the phase function $\Phi$ has quadratic growth with respect to $x,\o$ and the amplitude $b$ is in the H\"ormander class $S^0_{0,0}$, i.e., bounded together with its derivatives. The case of Hamiltonians $a\phas$ of polyhomogeneous type was first treated in  \cite{Chazarain,helffer84,helfferRobert80}. For the present general case, where homogeneity is not assumed, see e.g. \cite{asada-fuji,wiener3,wiener4,JMP2014,Fujiwara79,gz,Tataru04,Tataru08,Parmegg88,Treves}.

Time-frequency analysis enters the picture at this moment, giving a global expression for $e^{i t H}$. Let us recall the notation for the \tfs s:
\begin{equation}\label{tfs}
    \pi(z)f(t)= M_\o T_x f(t)=e^{2\pi i \la t,\eta\ra} f(t-x), \quad x,\o,t\in\rd, \,\,z=\phas.
\end{equation}
The  \stft (STFT) of a function or distribution $f$ on $\rd$ with respect to a Schwartz window function $g\in\cS(\rd)\setminus\{0\}$ is defined by
\begin{equation}\label{STFT}
    V_g f\phas =\la f, \pi(z) g\ra= \intrd f(v) \overline{g(v-x)}e^{-2\pi i \la v,\eta\ra } dv,\quad z=\phas\in\rdd.
\end{equation}

The \tf \, representation of a linear continuous operator $P:\cS(\rd)\to \cS'(\rd)$ is provided by the (continuous)
Gabor matrix
\begin{equation}\label{GbM}
k(w,z):=\la P\pi(w)g,\pi(z) g\ra,\quad w,z\in\rdd
\end{equation}
so that
\begin{equation}\label{KGbM}
V_g(Pf)(z)=\intrdd k(w,z) V_g f(w) \,dw.
\end{equation}
Following the pattern of \cite{CNRrevmathp}, we study the Gabor matrix $k(t,w,z)$ of the propagator $e^{i t H}$.
Its structure is linked with the Hamiltonian
field of $a(x,\xi)$.  Precisely, consider
\begin{equation}\label{1.11}
\begin{cases}
 2\pi\dot{x}=-\nabla  _\xi a ( x,\xi) \\
 2\pi \dot{\xi}=\nabla _x a (x,\xi)\\
 x(0)=y,\ \xi(0)=\eta,
\end{cases}
\end{equation}
(the factor $2\pi$ depends on the normalization of the STFT).
The solution
$\chi_t(y,\eta)=(x(t,y,\eta),\xi(t,y,\eta))$ exists for all
$t\in\bR$ thanks to our hypothesis, and defines a symplectic diffeomorphism
$\chi_t:\,\bR_{y,\eta}^{2d}\to\bR_{x,\xi}^{2d}$. The components of $\chi_t$ are functions with bounded smooth derivatives of any order in $\bR\times\rd$.
\begin{theorem}\label{T1.1}
Let  $k(t,w,z)$ be
the Gabor matrix of the Schr\"{o}dinger propagator $e^{i t H}$.
Then for every $s>0$ there exists $C=C(t,s)>0$ such that
\begin{equation}\label{KT1.1}
|k(t,w,z)|\leq C\la z-\chi_t(w)\ra^{-s},\quad
z=(x,\xi),\,w=(y,\eta)\in\rdd.
\end{equation}
\end{theorem}
 For $t$ small enough
our assumptions yield $\det \frac{\partial x}{\partial
y}(t,y,\eta)\not=0$ in the expression of $\chi_t$, and
\eqref{KT1.1} is then equivalent to \eqref{FIO1}
with the phase $\Phi$ linked to $\chi_t$ as standard and
$b(t,\cdot)\in S^0_{0,0}$, see the next Section $2$. In the
classical approach, cf.\ \cite{asada-fuji}, the occurrence of caustics
makes the validity of \eqref{FIO1} local in time and for global time $t\in
\bR$  multiple compositions of local representations are used,
with unbounded number of variables possibly appearing in the
expression. Observe instead that $k(t,w,z)$ keeps meaning  for every
$t\in\bR$, and the estimates \eqref{KT1.1} hold for $\chi_t$ with
$t\in\bR$.

A natural
functional frame to express boundedness and propagation results  for $e^{it H}$ is given by the modulation spaces, see \cite{F1} and the short survey in Section
\ref{2}.\par For  $1\leq p\leq \infty$, $r\in\bR$, the modulation space $M^p_r(\rd)$
is  defined by the space of all $f\in\cS'(\rd)$ for which
\begin{equation}\label{minfty}
\|f\|^p_{M^p_r(\rd)}=\intrdd |V_g f(z)|^p \la z\ra^{pr} dz<\infty
\end{equation}
(with obvious changes for $p=\infty$).\par
From Theorem \ref{T1.1} it is easy to deduce the following
\begin{theorem}\label{T1.2new}
For every $r\in\bR$, $1\leq p\leq \infty$, $t\in \bR$, we have
\begin{equation}\label{KT3.3}
e^{i t H}: M^p_r(\rd)\to M^p_r(\rd).
\end{equation}
\end{theorem}
The proofs of Theorems \ref{T1.1} and \ref{T1.2new} are given in \cite{CNRrevmathp} when the symbol $a\phas$ is a polyhomogeneous symbol. The proof for the present non-homogeneous situation follows closely the one in \cite{CNRrevmathp}. Our aim being to introduce non-expert readers to the methods of the \tfa, we shall reproduce in the sequel the main lines of the argument.\par
Novelty with respect to \cite{CNRrevmathp} will be a result of propagation of (micro) singularities for the solutions of \eqref{C1intro}. In fact, in \cite{CNRrevmathp} by taking advantage of the homogeneous structure, the propagation was expressed in terms of the global (Gabor) wave front set, whereas in
\eqref{C1intro}, \eqref{symbdec} homogeneity is lost in general. This will require the use of a more refined notion, namely the filter of the Gabor singularities, see below.

Let us first recall that the global wave front set was introduced by H{\"o}rmander \cite{hormanderglobalwfs91} in 1991, see also \cite{rw}, where the name of Gabor wave front set was given. The work of H{\"o}rmander \cite{hormanderglobalwfs91} was addressed to the study of the hyperbolic equations with double characteristics, and also provided propagation of singularities for Schr\"odinger equations \eqref{C1intro} with quadratic Hamiltonian $a(z)$. Such result was generalized to different classes of linear and nonlinear equations, beside \cite{CNRrevmathp} see for example \cite{nicola,RodinoW2015} and \cite{CNRanalitico2}, concerning the analytic category. One may find in these papers references to the wide previous literature on the subject. \par
The renewed and increasing interest for the Gabor wave front set derives from the fact that, under the action of a metaplectic operator, it moves according to the associated linear symplectic transformation. More generally, if $P$ is a Fourier integral operator as in \eqref{FIO1} with a phase function $\Phi$ homogeneous of degree $2$ in $\phas$, then the Gabor wave front set is determined by the corresponding map $\chi_t$.\par
Returning to the present non-homogeneous context, our definition of Gabor (micro) singularity will be as follows. Let $\Gamma$ be a subset of $\rdd$. Given a distribution $f\in\cS'(\rd)$, we say that $f$ is $M^p_r$ regular in $\Gamma$, $1\leq p\leq \infty$, $r\in\bR$, if there exists a neighborhood $\Gamma_\delta$ of $\Gamma$ (see precise definitions in Section $4$) such that
\begin{equation}\label{wVg}
   \int_{\Gamma_\delta} |V_g f(z)|^p \la z \ra ^{pr} dz <\infty.
\end{equation}
 We shall call \emph{filter of the $M^p_r$ singularities of $f$} the collection of subsets of $\rdd$
\begin{equation}\label{filter}
    \mathcal{F}^p_r(f)=\{ \Lambda \subset \rdd, \, f\, \mbox{is}\, M^p_r \,\mbox{regular\, in\,} \Gamma=\rdd\setminus\Lambda\}.
\end{equation}
Note that $f$ is regular in any bounded set $\Gamma\subset\rdd$. \par
We may now state our main result.
\begin{theorem}\label{T1.3new}
For every $r\in\bR$, $1\leq p\leq \infty$, $t\in \bR$, $u_0\in\cS'(\rd)$, we have
\begin{equation}\label{KT1.2}
\chi_t(\cF^p_r(u_0))=\cF_r^p(e^{i t H}u_0).
\end{equation}
\end{theorem}
The content of the next sections is the following. Section $2$ is devoted to some preliminaries, concerning properties of the \stft, modulation spaces and canonical transformations. Section $3$ presents the theory of the Fourier integral operators in terms of Gabor analysis, cf \cite{Wiener}. The proofs of Theorems \ref{T1.1} and \ref{T1.2new} are given as application. Section $4$ contains the analysis of the Gabor singularities and the proof of Theorem \ref{T1.3new}.

\vskip0.3truecm
\section{Preliminaries}\label{2}
 In what follows there are  the basic
concepts  of \tfa. For details we refer  to \cite{book}. We also discuss the properties of phase functions and canonical transformations.
\subsection{The Short-time Fourier Transform }\label{2.1}
Given a distribution $f\in\cS '(\rd)$ and a Schwartz function
$g\in\cS(\rd)\setminus\{0\}$ (the so-called {\it window}), the
short-time Fourier transform (STFT) of $f$ with respect to $g$ is defined as in \eqref{STFT}.
 The  \stft\ is well-defined whenever  the bracket $\langle \cdot , \cdot \rangle$ makes sense for
dual pairs of function or distribution spaces, in particular for $f\in
\cS ' (\rd )$ and $g\in \cS (\rd )$, $f,g\in\lrd$. Consider  $f,g\in\cS(\rd)$, then $V_gf\in\cS(\rdd)$.

\par
 We recall the following pointwise inequality for the \stft\
 \cite[Lemma 11.3.3]{book}, used  to change window functions.
 \begin{lemma}\label{changewind}
 If  $g_0,g_1,\gamma\in\cS(\rd)$ such
 that $\la \gamma, g_1\ra\not=0$ and
 $f\in\cS'(\rd)$,  then the inequality
 $$|V_{g_0} f(x,\xi)|\leq\frac1{|\la\gamma,g_1\ra|}(|V_{g_1} f|\ast|V_{g_0}\gamma|)(x,\xi)$$
 holds pointwise for all $(x,\xi)\in\rdd$.
 \end{lemma}

\subsection{Modulation spaces}\label{2.2}
Weighted modulation spaces measure the decay of the STFT on the time-frequency (phase space) plane. They were introduced by Feichtinger in the 80's \cite{F1}.

\emph{Weight Functions.}  A weight function $v$ is submultiplicative if $ v(z_1+z_2)\leq v(z_1)v(z_2)$, for all $z_1,z_2\in\Renn.$  We shall work with the weight functions
\begin{equation}\label{weight} v_s(z)=\la z\ra^s=(1+|z|^2)^{\frac s 2},\quad s\in\R,
\end{equation}
which are submultiplicative for $s\geq0$.\par
 For $s\geq0$, we denote by $\mathcal{M}_{v_s}(\rdd)$ the space of $v_s$-moderate weights on $\rdd$.  These  are measurable positive functions $m$ satisfying $$m(z+w)\leq C
v_s(z)m(w)$$ for every $z,w\in\rdd$.

\begin{definition}  \label{prva}
Given  $g\in\cS(\rd)$, $s\geq0$,  $m\in\mathcal{M}_{v_s}(\rdd)$, and $1\leq p,q\leq
\infty$, the {\it
  modulation space} $M^{p,q}_m(\Ren)$ consists of all tempered
distributions $f\in \cS' (\rd) $ such that $V_gf\in L^{p,q}_m(\Renn )$
(weighted mixed-norm spaces). The norm on $M^{p,q}_m(\rd)$ is
\begin{equation}\label{defmod}
\|f\|_{M^{p,q}_m}=\|V_gf\|_{L^{p,q}_m}=\left(\int_{\Ren}
  \left(\int_{\Ren}|V_gf(x,\xi)|^pm(x,\xi)^p\,
    dx\right)^{q/p}d\xi\right)^{1/q}  \,
\end{equation}
(obvious modifications if $p=\infty$ or $q=\infty$).
\end{definition}
 When $p=q$, we  write $M^{p}_m(\rd)$ instead of $M^{p,p}_m(\rd)$ and $M^{p}_s(\rd)$ for $M^{p}_{v_s}(\rd)$. The spaces $M^{p,q}_m(\rd)$ are Banach spaces and every nonzero $g\in M^{1}_{v_s}(\rd)$ yields an equivalent norm in \eqref{defmod} and hence $M^{p,q}_m(\Ren)$ is independent on the choice of $g\in  M^{1}_{v_s}(\rd)$.

We recover the H\"ormander
class
\begin{equation}\label{HC}S^0_{0,0}=\bigcap_{s\geq 0}M^{\infty}_{1\otimes v_s}(\rdd).\end{equation}
For any $1\leq p,q\leq\infty$,

\begin{equation}\label{HCbis}\bigcap_{s\geq 0}M^{p,q}_{ v_s}(\rd)=\cS(\rd),\quad \bigcup_{s\geq 0}M^{p,q}_{ v_{-s}}(\rd)=\cS'(\rd).\end{equation}

Fix $g\in\cS(\rd)\setminus\{0\}$. The adjoint operator of $V_g$,  defined by $\la V_g^\ast F, h\ra=\la F,V_g h\ra$, can be written as
 \begin{equation}\label{adj}V_g^\ast F=\intrdd F(x,\xi) \pi(x,\xi) g dx d\xi,
\end{equation}
 The adjoint $ V_g^\ast$ maps  the Banach space $L^{p,q}_m(\rdd)$ into $M^{p,q}_m(\rd)$, in particular it maps $\cS(\rdd)$ into $\cS(\rd)$ and the same for their dual spaces. If $F=V_g f$ we obtain the  inversion formula for the STFT
 \begin{equation}\label{treduetre}
 {\rm Id}_{M^{p,q}_m}=\frac 1 {\|g\|_2^2} V_g^\ast V_g
 \end{equation}
(the same holds when replacing $M^{p,q}_m(\rd)$ by $\cS(\rd)$ or $\cS'(\rd)$).\par

\subsection{Phase functions and canonical transformations}\label{2.3}
Let $a$ be as in the Introduction, real-valued and satisfying \eqref{symbdec}. The related classical evolution, given by the linear Hamilton-Jacobi
system, following our normalization can be written as
\begin{equation}\label{HS}
\begin{cases}
 2\pi \partial_t x(t,y,\eta)=-\nabla  _\xi a ( x(t,y,\eta),\xi(t,y,\eta)) \\
 2\pi \partial_t\xi(t,y,\eta)=\nabla _x a (x(t,y,\eta),\xi(t,y,\eta))\\
 x(0,y,\eta)=y,\\
 \xi(0,y,\eta)=\eta.
\end{cases}
\end{equation}

The solution $(x(t,y,\eta),\xi(t,y,\eta))$ exists for every $t\in\bR$. Indeed,
setting $u:=(x,\xi)$, $F(u):=(-\nabla_\xi a(u), \nabla _x a(u))$,
 the initial value problem \eqref{HS} can be rephrased as
\begin{equation}\label{IVP}
u^\prime (t)=F(u(t)),\quad u(t_0)=u_0,
\end{equation}
in the particular case $t_0=0$.
Observe that our assumptions on $a$  imply the boundedness of  $\partial^\a F_j$, for every $|\a|>0$,  $j=1,\dots,2d$, hence in particular $F: \rdd\to \rdd$ is a Lipschitz continuous mapping. The previous ODE  is an autonomous ODE with a mapping $F\in\cC^\infty(\rdd\to\rdd)$ having at most linear growth, hence $\|F(u)\|\lesssim 1+\|u \|$. Hence for each $u_0\in\rdd$ and $t_0\in\bR$ there exists a unique classical global solution $u\,:\, \bR\to \rdd$ (in this case  $u\in \cC^\infty(\bR\to\rdd)$ since $F\in\cC^\infty(\rdd\to\rdd)$)  to \eqref{IVP}. The solution maps $S_{t_0}(t)\,:\rdd\to\cC^\infty(\bR\to\rdd)$,  defined by $S_{t_0}(t)u_0=u(t)$, and  $S_{t_0}(t_0)={\rm Id}$, the identity operator on $\rdd$, are  Lipschitz continuous mappings, obeying  $S_{t_0}(t)=S_0(t-t_0)$ and the group laws
\begin{equation}\label{prodotto}
S_{0}(t)S_0(t')=S_0(t+t'),\quad S_0(0)={\rm Id}.
\end{equation}
The mapping $S_{0}(t)$ is a bi-Lipschitz diffeomorphism with $S_{0}^{-1}(t)=S_{0}(-t)$.
Following the notations of  \cite{Wiener,CNRrevmathp},  we call the bi-Lipschitz diffeomorphism
\begin{equation}\label{mappachi}\chi_t(y,\eta):=S_0(t)(y,\eta),\quad (y,\eta)\in\rdd.\end{equation}

The theory of Hamilton-Jacobi allows to find a $T>0$ such that for $t\in ]-T,T[$ there exists a phase function  $\Phi(t,x,\eta)$, solution of the eiconal equation 
\begin{equation}\label{eiconal}
\begin{cases}
2\pi \partial_t\Phi+a (x,\nabla_x\Phi)=0\\
\Phi(0,x,\eta)=x \eta
\end{cases}
\end{equation}
The phase $\Phi(t,x,\eta)$ is real-valued since the 
symbol $a(x,\xi)$ is real-valued, moreover $\Phi$ fulfills  the condition
of non-degeneracy:
\begin{equation}\label{detmisto}
|\det \partial_{x,\eta}^2 \Phi(t,x,\eta)|\geq c>0,\quad
(t,x,\eta)\in ]-T,T[\times \rdd,
\end{equation}
after possibly shrinking $T>0$,  and satisfies
\begin{equation}\label{39bis}
    |\partial^k_t\partial^\a_{x,\o}\Phi(t,x,\o)|\leq c_{k,\a},\quad |\a|\geq 2,\,k\geq0,\,(t,x,\o)\in ]-T,T[\times \rdd.
\end{equation}
 The relation between the phase $\Phi$ and the canonical transformation $\chi$ is given by
\begin{equation}\label{rel-chi-phi}
(x,\nabla_x\Phi(t,x,\eta))=\chi_t(\nabla_\eta\Phi(t,x,\eta),\eta),\quad t\in ]-T,T[.
\end{equation}
In particular,
\begin{equation}\label{cantra} \left\{
                \begin{array}{l}
                y(t,x,\eta)=\nabla_{\eta}\Phi(t,x,\eta)
                \\
               \xi(t,x,\eta)=\nabla_{x}\Phi(t,x,\eta), \rule{0mm}{0.55cm}
                \end{array}
                \right.
\end{equation}
and there exists $\delta>0$ such that
\begin{equation}\label{detcond2}
   |\det\,\frac{\partial x}{\partial y}(t,y,\eta)|\geq \delta, \quad t\in ]-T,T[.
\end{equation}
Observe that each component of $\chi_t$ is a function with bounded smooth derivatives of any order in $]-T,T[\times \rdd$.  Using \eqref{prodotto} we observe that the same holds in fact for every $t\in\bR$.

For $t\in ]-T,T[$, the  phase function $\Phi(t,\cdot)$  is  a \emph{tame} phase,  and similarly for the canonical transformation $\chi_t$, according to the following definition \cite[Definition 2.1]{Wiener}:\par

\emph{A real and smooth phase function $\Phi(x,\eta)$ on $\rdd$ is called \emph{tame} if}:\\
(i) \emph{For} $z=\phas$,
\begin{equation}\label{phasedecay}
|\partial_z^\a \Phi(z)|\leq C_\a,\quad |\a|\geq 2;\end{equation}
(ii) \emph{There exists $c>0$ such that} 
\begin{equation}\label{detcond}
   |\det\,\partial^2_{x,\eta} \Phi(x,\o)|\geq c.
\end{equation}
 
 \par \emph{The mapping  defined by $(x,\xi)=\chi(y,\o)$, which solves the system}
 \begin{equation}\label{cantra2} \left\{
                 \begin{array}{l}
                 y(x,\eta)=\nabla_{\eta}\Phi(x,\eta)
                 \\
                \xi(x,\eta)=\nabla_{x}\Phi(x,\eta), \rule{0mm}{0.55cm}
                 \end{array}
                 \right.
 \end{equation}
\emph{ is called \emph{tame} canonical transformation.}\par

Observe that we have no assumption of homogeneity for large $(x,\eta)$, nevertheless the mapping $\chi$ is well-defined by the global inverse
function theorem. The mapping  $\chi$ is a smooth bi-Lipschitz canonical transformation (i.e.\ it preserves the symplectic form)  and satisfies, for $(x,\xi)=\chi(y,\o)$,
\begin{equation}\label{B2}
|\partial_{z}^\a x_i(z)|+|\partial_{z}^\a \xi_i(z)|\leq C_\a,\quad |\a|\geq 1,\,\,z=(y,\eta),\,\,i=1,\dots,d.\end{equation}
The mapping $\chi$ enjoys  \begin{equation}\label{detcond2?}
    |\det\,\frac{\partial x}{\partial y}(y,\eta)|\geq \delta
 \end{equation}
 (that is \eqref{detcond2} for the canonical transformations of the Hamilton-Jacobi theory), which allows to uniquely determine, up to a constant, the related tame phase function $\Phi_\chi$ (see \cite[Section 2]{Wiener}).

\section{Fourier Integral Operators}
\subsection{The classes $FIO(\chi)$}
The class $FIO(\chi)$ was introduced in \cite{Wiener} and its definition can be rephrased as follows.

\begin{definition}
Let $g\in\cS(\rd)$ be a non-zero window function. Consider   a canonical transformation $\chi$ which is a smooth bi-Lipschitz diffeomorphism and satisfies \eqref{B2}. A
continuous linear operator $T:\cS(\rd)\to\cS'(\rd)$ is in the
class $FIO(\chi)$ if its (continuous) Gabor matrix  satisfies for all $s>0$ the decay
condition
\begin{equation}\label{asterisco}
|\langle T \pi(w) g,\pi(z)g\rangle|\leq {C}\langle z-\chi(w)\rangle^{-s},\qquad \forall z,w\in\rdd,
\end{equation}
for a constant $C>0$ depending on $s$.
\end{definition}
Note that we do not require \eqref{detcond2?} to be valid.\par
The class $\fiola = \bigcup _{\chi } FIO (\chi)$ is  the union of
these classes where  $\chi$ runs over the  set of all  smooth bi-Lipschitz canonical transformations satisfying \eqref{B2}.

The following properties are proved in \cite{Wiener}.
\begin{itemize}
    \item [(i)] \emph{Boundedness of $T$ on $M^p(\rd)$} (\cite[Theorem 3.4]{Wiener}):\\
    If $T\in FIO(\chi)$, then $T$ can be extended to a bounded operator on $M^p(\rd)$ (in particular on $\lrd$).
    \item[(ii)] \emph{The algebra property} (\cite[Theorem 3.6]{Wiener}): For $ i=1,2$, 
\begin{equation}\label{algebra}T^{(i)}\in FIO(\chi_i)\quad \Rightarrow \quad T^{(1)}T^{(2)}\in
FIO(\chi_1\circ \chi_2) \, .\end{equation}
 %with $s=\min(s_1,s_2)$.
 \item[(iii)] \emph{The Wiener property} (\cite[Theorem 3.7]{Wiener}): If $T\in FIO(\chi)$  and  $T$ is invertible on $L^2(\rd)$, then $T^{-1} \in FIO(\chi^{-1})$.
\end{itemize}
These  properties imply that the union $\fiola$ is a Wiener subalgebra of $\cL
(\lrd )$, the class of linear bounded operators on $\lrd$. Property (ii) can be refined as follows.
\begin{lemma} For  $T^{(i)}\in FIO(\chi_i)$, $ i=1,2$, the continuous Gabor matrix of the composition $T^{(1)}T^{(2)}$ is controlled for every $s\geq 0$ by
\begin{equation}\label{controllo1}
|\langle T^{(1)}  T^{(2)}\pi(w) g,\pi(z)g\rangle|\leq C_0 C_1 C_2 \la z-\chi_1\circ\chi_2(w)\ra^{-s}, w,z\in\rdd,
\end{equation}
where $C_i>0$ is the constant  of $T^{(i)}$ in \eqref{asterisco}, $ i=1,2$, whereas $C_0>0$ depends only on $s$ and on the Lipschitz constants of $\chi_1$ and $\chi_1^{-1}$.
\end{lemma}
The proof is an easy consequence of Lemma 2.11 in \cite{CNRrevmathp}.

By induction we immediately obtain
\begin{proposition} For $n\in\bN$, $n\geq2$,  $T^{(i)}\in FIO(\chi_i)$,  $i=1,\dots,n$, we have
\begin{equation}\label{GMn}
|\langle T^{(1)} T^{(2)}\cdots  T^{(n)}\pi(w) g,\pi(z)g\rangle|\leq C_0C_1\cdots C_n \la z-\chi_1\circ\chi_2\circ\cdots \circ\chi_{n}(w)\ra^{-s}.
\end{equation}
where $C_0$ depends on $s$ and on the Lipschitz constants of the mappings: $$\chi_1,\chi^{-1}_1,\chi_1\circ\chi_2,(\chi_1\circ\chi_2)^{-1},\dots,\chi_1\circ\chi_2\circ\cdots\circ\chi_{n-1},(\chi_1\circ\chi_2\circ\cdots\circ\chi_{n-1})^{-1}.$$
\end{proposition}
\medskip

Observe that, using Schur's test and the same techniques as in the proof \cite[Theorem 3.4]{Wiener}, it is straightforward to obtain the following
weighted version of \cite[Theorem 3.4]{Wiener}.
\begin{theorem}\label{contmp} Consider
$\mu\in\cM_{v_r}$. Then for every
$1\leq p\leq\infty$, $T\in FIO(\chi)$ extends
to a continuous operator from
${M}^p_{\mu\circ\chi}$ into
${M}^p_{\mu}$.
\end{theorem}
Let us underline  that $\mu\circ\chi\in\cM_{v_r}$, since
 $v_r\circ\chi\asymp v_r$,  due to the bi-Lipschitz property of $\chi$. In particular
 
 \begin{equation}\label{52bis}
    T: M^p_r(\rd)\to M^p_r(\rd), \quad 1\leq p\leq \infty, r\in \bR.
\end{equation}
If $\chi=\mathrm{Id}$,  then the corresponding Fourier integral operators are simply pseudodifferential operators, as already shown in \cite{GR}. \par
 The characterization below, written for \psdo s in the Kohn-Nirenberg form $\sigma(x,D)$, works  for any  $\tau$-form (in particular Weyl form $\sigma^w(x,D)$)  of a pseudodifferential operator.
 \begin{proposition}
   \label{charpsdo}
 Fix $g \in\cS(\rd)$ and let $\sigma \in \cS'(\rdd )$. For $s\in\bR$, the symbol $\sigma$ belongs to $S^0_{0,0}(\rdd )$ \fif\ for every $s\geq 0$, and for suitable constants $C=C_s$ depending on $s$,
 \begin{equation}
   \label{eq:kh9}
   |\langle \sigma(x,D) \pi (w) g , \pi (z) g\rangle | \leq C_s \la z-w\ra^{-s}, \qquad
   \forall \,w,z \in \rdd \, .
 \end{equation}
 \end{proposition}
 \par
 Similarly, under additional assumptions on  the classes $FIO(\chi)$, their operators can be written in the following integral form (\emph{FIOs of type I}):
 \begin{equation}\label{sei}
 I(\sigma,\Phi) f(x)=\int_{\rd} e^{2\pi i
  \Phi(x,\eta)}\sigma(x,\eta)\widehat{f}(\eta)\,d\eta,  \quad f\in\cS(\rd),
 \end{equation}
 where  $\sigma\in S^0_{0,0}(\rdd)$) and $\Phi$ a tame phase function. More precisely, this particular form is allowed starting from the class $FIO(\chi)$ whenever  the mapping $\chi$ enjoys the additional property  \eqref{detcond2?} as explained in the following characterization \cite[Theorem 4.3]{Wiener}.

\begin{theorem}\label{caraI}
Consider $g\in\cS(\rd)$.
Let $I$ be a continuous linear operator $\cS(\rd)\to\cS'(\rd)$ and
$\chi$ be a tame canonical transformation satisfying \eqref{detcond2?}.  Then the following properties are
equivalent. \par {\rm (i)} $I=I(\sigma,\Phi_\chi)$ is a FIO of type
I for some $\sigma\in S^0_{0,0}(\rdd)$. \par {\rm
(ii)} $I\in FIO(\chi)$.\par
\end{theorem}

For $\chi={\rm Id}$ we recapture the characterization for \psdo s of Proposition \ref{charpsdo}.

\begin{remark}\label{3.7}
We shall apply the preceding  results to the mappings $\chi_t(x,\eta)$ coming from the Hamilton-Jacobi system \eqref{HS}, so we need to be more precise on the estimate \eqref{asterisco}, namely we have to show how  the constants $C$  depend on the time variable $t$.
To this end, in short: gluing together the results \cite[Theorem 3.3]{locNC09} and \cite[Theorem 4.3]{Wiener}, we obtain that the constants $C=C_s$ in \eqref{asterisco}, $s\geq 0$, can be estimated in terms of the $S^0_{0,0}$-semi-norms of $\sigma$. Provided the Lipschitz constants of $\chi_t$ are uniformly bounded with respect to $t$, as we have in  \eqref{rel-chi-phi}, we may assume $C_s(t)\in \cC(]-T,T[)$ in \eqref{asterisco} if the semi-norms of $\sigma_t$ can be estimated by functions in 
$\cC(]-T,T[)$.
\end{remark}

\section{Proof of Theorems \ref{T1.1} and \ref{T1.2new}}

Let us consider the Cauchy problem \eqref{C1intro} with $a(z)$ real-valued satisfying \eqref{symbdec}. From the results recalled in the Introduction we have for $e^{i t H}$ the representation \eqref{FIO1}. More precisely:

\begin{proposition}\label{Prop31}
There exists a constant $T>0$, a symbol $\sigma(t,x,\eta)\in \cC^\infty(]-T,T[,S^{0}_{0,0})$ a real-valued phase function $\Phi$ satisfying
\eqref{eiconal} and \eqref{detmisto} such that the evolution operator can be written as
 \begin{equation}\label{soluzioneA}
 (e^{i t H}u_0)(t,x)=(F_tu_0)(t,x)
 \end{equation}
 where $F_t$ is the FIO of type I
 \begin{equation}\label{FIO1bis}
 (F_tu_0)(t,x)=\intrd e^{2\pi i\Phi(t,x,\eta)} \sigma(t,x,\eta){\widehat
  {u_0}}(\eta)d\eta.
  \end{equation}
\end{proposition}

\begin{remark}\label{rnew}
Notice that the function $\Phi(t,\cdot)$ of Proposition \ref{Prop31} and the related canonical transformation $\chi_t$ in \eqref{mappachi} are tame. The Lipschitz constants of $\chi_t$ and $\chi_t^{-1}$  can be controlled by a continuous function of $t$ on the interval $]-T,T[$ and thus can be chosen uniform with respect to $t$ on $]-T,T[$.\par
\end{remark}
From Theorem \ref{caraI} we infer that the propagator $e^{it H}$ belongs to $\fiola$ for every fixed $t\in ]-T,T[$. More precisely:
\begin{proposition}\label{cor31}
Under the assumptions of Proposition \ref{Prop31} we have \begin{equation}
\label{T1}
e^{i t H}\in FIO(\chi_t),\quad t\in]-T,T[,\end{equation}
where  $\chi_t$ is defined in \eqref{mappachi}. Moreover for every $s\geq0$ there exists $C(t)=C_s(t)\in \cC(]-T,T[)$ such that, for every $g\in\cS(\rd)$ the Gabor matrix satisfies
\begin{equation}\label{GM1}
|\la e^{i t H} \pi(w)g, \pi(z)g\ra|\leq C(t)  \la z-\chi_t(w)\ra^{-s},\quad w,z\in\rdd.
\end{equation}
\end{proposition}

The last part of the statement follows from Remark \ref{rnew} and Remark \ref{3.7}.

The previous proposition gives a representation  of $e^{it H}$ for $|t|<T$. Using the group property of the propagator $e^{it H}$  we may obtain an expression of $e^{it H}$ for every $t\in\bR$.
Indeed, a classical trick, jointly with the group property of   $e^{it H}$, applies. Namely, we consider $T_0<T/2$ and define
$$I_h=]h T_0, (h+2)T_0[, \quad h\in \bZ.
$$
For $t\in I_h$, by the group property of  $e^{it H}$:
\begin{equation}\label{eita}
e^{it H}=e^{i(t-hT_0)H}(e^{i(hT_0) H/|h|})^{|h|}
\end{equation}
and using Proposition \ref{Prop31}, one can write
\begin{equation}\label{eita2}e^{it H}=F_{t-hT_0}(F_{\frac{h}{|h|}T_0})^{|h|}.
\end{equation}
In general, $e^{it H}$ or even the composition  $F_{t-hT_0}(F_{\frac{h}{|h|}T_0})^{|h|}$ cannot be represented as a type I FIO in the form \eqref{FIO1}.  We shall prove below that
the evolution $e^{it H}$ is in the class  $FIO(\chi_t)$ for every $t\in\bR$, with $\chi$  defined in \eqref{mappachi}, so that this class is proven to be the right framework for describing the evolution  $e^{it H}$.

\begin{theorem}\label{UNCP}
Given the  Cauchy problem \eqref{C1intro} with $a(z)$ real valued satisfying \eqref{symbdec}, consider the mapping   $\chi_t$ defined in \eqref{mappachi}.
 Then  
 \begin{equation}\label{A1}e^{it H}\in  FIO(\chi_t),\quad t\in\bR\end{equation}
 and for every $s\geq 0$ there exists $C(t)\in \cC(\bR)$ such that
 \begin{equation}\label{GM2}
 |\la e^{i t H} \pi(w)g, \pi(z)g\ra|\leq C(t)  \la z-\chi_t(w)\ra^{-s},\quad w,z\in\rdd,\quad t\in\bR.
 \end{equation}
\end{theorem}
\begin{proof} We fix $T_0<T/2$ as above. For $t\in\bR$, there exists  $h\in\bZ$ such that $t\in I_h$. Using Proposition \ref{cor31} for $t_1=t-hT_0\in ]-T,T[$ we have that $e^{it_1H}\in FIO(\chi_{t_1})$ and for $t_2=\frac{h}{|h|}T_0 \in ]-T,T[$,
$e^{it_2 H}\in FIO(\chi_{t_2})$, and for every $s\geq 0$, there exists a continuous function $C(t)$ on $]-T,T[$ such that \eqref{GM1} is satisfied for $t=t_1$ and $t=t_2$. Using the algebra property \eqref{algebra}, we have
$$e^{it_1H}(e^{it_2H})^{|h|}\in FIO(\chi_{t_1}\circ(\chi_{t_2})^{|h|})$$
and  the group law \eqref{prodotto} for $\chi_t(y,\eta)=S_0(t)(y,\eta)$ gives
$$
\chi_{t_1}\circ(\chi_{t_2})^{|h|}=\chi_{t_1+|h| t_2}=\chi_{t},
$$
as expected. Then, using \eqref{GMn} we obtain that the Gabor matrix of the product $e^{it_1H}(e^{it_2H})^{|h|}$ is controlled by a continuous function $C_h(t) $ on $I_h$. Finally, from  the estimates
$$ |\la e^{it H} \pi(w)g, \pi(z)g\ra|\leq C_h(t) \la z-\chi_t(w)\ra^{-s},\quad t\in I_h,$$
with $C_h\in\cC(I_h)$, it is easy to construct a new continuous controlling function $C(t)$ on $\bR$ such that \eqref{GM2} is satisfied.
\end{proof}

In particular,  the estimate \eqref{GM2} gives Theorem \ref{T1.1}.  Using Theorem \ref{contmp} we obtain Theorem \ref{T1.2new}.

\section{Gabor singularities and proof of Theorem \ref{T1.3new}}

We want now to localize in $\rdd$ the Gabor singularities of a distribution and study the action on them of $e^{itH}$.\par
For $\Gamma\subset\rdd$ we define the $\delta$-neighborhood $\Gamma_\delta$, $0<\delta<1$, as 
\begin{equation}\label{7.1}
\Gamma_\delta=\{z\in\rdd:\ |z-z_0|<\delta\langle z_0\rangle\ \textrm{for some}\ z_0\in\Gamma\}.
\end{equation}
We begin to list some properties of the $\delta$-neighborhoods, for the proofs we refer to \cite[Lemmas 7.1, 7.2]{CNRanalitico2}.
\begin{lemma}\label{lemma7.1}
Given $\delta$, we can find $\delta^\ast$, $0<\delta^\ast<\delta$, such that for every $\Gamma\subset\rdd$
\begin{equation}\label{eq7.2}
\big(\Gamma_{\delta^\ast}\big)_{\delta^\ast}\subset\Gamma_{\delta},
\end{equation}
\begin{equation}\label{eq7.3}
\big(\rdd\setminus\Gamma_{\delta}\big)_{\delta^\ast}\subset \rdd\setminus\Gamma_{\delta^\ast}.
\end{equation}
\end{lemma}

\begin{lemma}\label{lemma7.2}
Let $\chi$ be a smooth bi-Lipschitz canonical transformation as in the preceding sections. For every $\delta$ there exists $\delta^\ast$, $0<\delta^\ast<\delta$, such that for every $\Gamma\subset\rdd$ 
\begin{equation}\label{eq7.4}
\chi(\Gamma_{\delta^\ast})\subset\chi(\Gamma)_\delta,
\end{equation}
\begin{equation}\label{eq7.5}
\chi(\Gamma)_{\delta^\ast}\subset\chi(\Gamma_\delta).
\end{equation}
The constant $\delta^\ast$ depends on $\chi$ and $\delta$ but it is independent of $\Gamma$.
\end{lemma}
In the following we shall argue on $f\in \cS'(\rd)$, and take windows $g\in \cS(\rd)$.

Since $\cup_{s\geq 0}M^p_{-s}(\rd)=\cS'(\rd)$ we have for some $s_0\geq 0$
\begin{equation}\label{eq7.6}
\intrdd |V_g f(z)|^p \la z\ra ^{-p s_0}dz<\infty.
\end{equation}

\begin{definition}\label{def7.3}
Let $f\in \cS'(\rd)$, $g\in \cS(\rd)\setminus\{0\}$, $\Gamma\subset\rdd$, $1\leq p\leq\infty$, $r\in\bR$. We say that $f$ is $M^p_r$-regular in $\Gamma$ if there exists  $\delta>0$ such that 
\begin{equation}\label{eq7.7}
\int_{\Gamma_\delta}|V_g f(z)|^p\la z\ra^{pr} dz<\infty.
\end{equation}
(obvious changes if $p=\infty$).
\end{definition}
Of course, \eqref{eq7.7} gives us some nontrivial information about $f$ only when $\Gamma$ is unbounded. We shall prove later that Definition \ref{def7.3} does not depend on the choice of the window $g\in \cS(\rd)$. \par
\begin{theorem}\label{teo7.4}
Let $e^{it H}$ and $\chi_t$ be defined as in the previous Sections, fix $u_0\in \cS'(\rd)$ and $\Gamma\subset\rdd$. If $u_0$ is $M^p_r$-regular in $\Gamma$, then $e^{it H}u_0$ is $M^p_r$-regular in $\chi_t(\Gamma)$.
\end{theorem}
\begin{proof}
For sufficiently small $\delta>0$  we have \eqref{eq7.7} in $\Gamma_\delta$ whereas \eqref{eq7.6} is valid in $\rdd$ for some $s_0\geq 0$. Now, from Theorem \ref{T1.1} we have 
\begin{equation}\label{eq7.8}
V_g(e^{it H} u_0)(z)=\int k(t,w,z) V_gu_0(w)\, dw
\end{equation}
with 
\begin{equation}\label{eq7.9}
|k(t,w,z)|=|\langle e^{i t H} \pi(w) g,\pi(z)g\rangle|\lesssim \la z-\chi_t(w)\ra^{-s},
\end{equation}
 for every $s\geq0$.\par
 We want to show that $ e^{i t H}u_0$ is  $M^p_r$-regular in $\chi_t(\Gamma)$. To this end, using \eqref{eq7.5} in Lemma \ref{lemma7.2}, we take first $\delta^\ast<\delta$ such that $\chi_t(\Gamma)_{\delta^\ast}\subset\chi_t(\Gamma_\delta)$ and then using \eqref{eq7.2} in Lemma \ref{lemma7.1} we fix $\delta'<\delta^\ast$ such that 
 \begin{equation}\label{eq7.10}
 \big(\chi_t(\Gamma)_{\delta'}\big)_{\delta'}\subset\chi_t(\Gamma)_{\delta^\ast}\subset \chi_t(\Gamma_\delta).
 \end{equation}
 Note that for $w\not\in\Gamma_{\delta}$, i.e.\ $\chi_t(w)\not\in \chi_t(\Gamma_{\delta})$, and $z\in \chi_t(\Gamma)_{\delta'}$ we have 
 \begin{equation}\label{eq7.11}
 |z-\chi_t(w)|\gtrsim \max\{\langle z\rangle,\langle w\rangle\}
 \end{equation}
 since $\chi_t(w)\not\in\big((\chi_t(\Gamma)_{\delta'}\big)_{\delta'}$ in view of \eqref{eq7.10}, and we may use as well \eqref{eq7.3}.\par 
 Assuming for simplicity $p<\infty$, we shall prove 
 \begin{equation}\label{eq7.12}
 \int_{\chi_t(\Gamma)_{
  \delta'}}|V_g(e^{i t H}u_0)(z)|^p\la z\ra^{pr}\, dz<\infty,
 \end{equation}
  with $\delta'$ determined as before. Using \eqref{eq7.8} and \eqref{eq7.9}, we  estimate
\begin{equation}\label{5.8eq}
|\la z\ra^r V_g (e^{itH} u_0)(z)|\lesssim \intrdd I(z,w) \,dw,
\end{equation}
with
\begin{equation}\label{eq5.9}
I(z,w)= \la z \ra^r\la z-\chi_t(w)\ra^{-s}|V_g u_0(w)|.
\end{equation}
To show that $e^{itH} u_0$ is $M^p_r$-regular in $\chi_t (\Gamma)$ it will be sufficient to show that $$ \left\| \intrdd I(\cdot,w)\,dw\right\|_{L^p(\chi_t(\Gamma)_{\delta'})}<\infty.$$
First, we estimate $\intrdd I(z,w)\,dw$ for $z\in \chi_t(\Gamma)_{\delta'}$. We split the domain of integration into two domains $\Gamma_\delta$ and $\rdd\setminus \Gamma_\delta$. In $\rdd\setminus \Gamma_\delta$ we   use \eqref{eq7.11} to obtain
\begin{align*}\int_{\rdd\setminus \Gamma_\delta} I(z,w)\,dw &\leq \int_{\rdd\setminus \Gamma_\delta} \la z \ra^r\la w \ra ^{s_0}\la z-\chi_t(w)\ra^{-s}\frac{|V_g u_0(w)|}{\la w \ra ^{s_0}}\, dw\\
&\lesssim  \int_{\rdd\setminus \Gamma_\delta} \la z-\chi_t(w)\ra^{r+s_0-s}\frac{|V_g u_0(w)|}{\la w \ra ^{s_0}}\, dw\\
&\lesssim \left(\la \cdot\ra ^{r+s_0 -s}\ast \frac{|V_g u_0(\cdot)|}{\la \cdot \ra ^{s_0}}\right)(z).
\end{align*}
So by \eqref{eq7.6} and choosing $s$ in \eqref{eq7.9} so that $r+s_0-s<-2d$,
$$\left\|\int_{ \rdd\setminus \Gamma_\delta} I(\cdot,w)\,dw \right \|_{L^p(\chi_t(\Gamma)_{\delta'})}\lesssim \|\la \cdot\ra ^{r+s_0 -s}\|_{L^1(\rdd)}\|\,|V_g u_0|\la \cdot \ra ^{-s_0}\|_{L^p(\rdd)}<\infty.$$
 In the domain $\Gamma_\delta$,
we have
\begin{align*}\int_{ \Gamma_\delta} I(z,w)\,dw &\leq \int_{\Gamma_\delta} \la z \ra^r\la w \ra ^{-r}\la z-\chi_t(w)\ra^{-r}\la z-\chi_t(w)\ra^{r-s}|V_g u_0(w)|\la w \ra ^r\, dw\\
&\lesssim  \int_{ \Gamma_\delta} \la z-\chi_t(w)\ra^{r-s}|V_g u_0(w)|\la w \ra ^r\, dw\\
&\lesssim \la \chi_t^{-1}(\cdot)\ra ^{r -s}\ast \left(\mbox{Char}_{\Gamma_\delta}\cdot|V_g u_0
|\la \cdot \ra ^r\right)(z)
\end{align*}
where $\mbox{Char}_{\Gamma_\delta}$ is the characteristic function of the set $\Gamma_\delta$. The assumption  \eqref{eq7.7} yields to the estimate
\begin{align*}\left\|\int_{ \Gamma_\delta} I(\cdot,w)\,dw \right \|_{L^p(\chi_t(\Gamma)_{\delta'})}&\lesssim \|\la \chi_t^{-1}(\cdot)\ra ^{r -s}\|_{L^1(\rdd)}\| \,|V_g u_0|\la \cdot \ra ^r\|_{L^p(\Gamma_\delta)}\\
&\asymp \|\la \cdot\ra ^{r -s}\|_{L^1(\rdd)}\| V_g u_0\la \cdot \ra ^r\|_{L^p(\Gamma_\delta)} <\infty,
\end{align*}
because $\chi_t$ is a bi-Lipschitz diffeomorphism and we may take $s$ in \eqref{eq7.9} so large that $r-s<-2d$.
This concludes the proof.
\end{proof}
\par
\begin{remark}
It is now easy to check that Definition \ref{def7.3} does not depend on the choice of the window $g\in\cS(\rd)$. In fact, assume that the estimate \eqref{eq7.7} in Definition \ref{def7.3} is satisfied for some  $\delta>0$, and some choice of $g\in \cS(\rd)$. Then \eqref{eq7.7} is still satisfied, for some new  $\delta>0$, if we replace $g$ with $h\in \cS(\rd)$.  To prove this claim, observe that
 for every $s\geq0$,
 \[
 |V_h g(z)|\lesssim \la z \ra^{-s},\quad z\in \rdd.
 \]
 Lemma \ref{changewind} then gives
 \[
 |V_h f(w)|\lesssim\int \la w-z\ra^{-s}|V_g f(z)|\, dz.
 \]
 The claim follows by splitting  the domain of integration into $\Gamma_\delta$ and $\rdd\setminus\Gamma_\delta$, and arguing as in the proof of the preceding Theorem \ref{teo7.4}, being now $\chi_t={\rm identity}$. 
\end{remark}

 The following definition allows one to describe the position in phase space of the singularities of a function $f$.
 \begin{definition}\label{def7.7}
 Given $f\in \cS'(\rd)$, we shall call filter of the Gabor singularities of $f$ the collection of subsets of $\rdd$: 
 \[
 \Fu^p_r(f)=\{\Lambda\subset\rdd:\ f\ \textit{is}\,  M^p_r\textit{-regular\, in}\ \Gamma=\rdd\setminus \Lambda\},
 \]
 cf.\ Definition \ref{def7.3}. 
 \end{definition}
 
 Let us write $\Fu(f)= \Fu^p_r(f)$ for short.
 
 $\Fu(f)$ is a filter since if $\Lambda\in \Fu(f)$ and $\Lambda\subset\Lambda'$, then also $\Lambda'\in \Fu(f)$, and moreover if $\Lambda_1,\ldots,\Lambda_n\in \Fu(f)$ then also $\cap_{j=1}^n \Lambda_j\in \Fu(f)$. Note that any neighborhood of $\infty$ i.e.\ the complementary of a bounded set, belongs to $\Fu(f)$. We have $f\in \cS(\rd)$ if and only if $\emptyset\in \Fu(f)$, that is equivalent to saying that there exists $\Lambda_1,\ldots,\Lambda_n\in \Fu(f)$ such that $\cap_{j=1}^n \Lambda_j=\emptyset$.\par
 Theorem \ref{teo7.4}  now follows. Indeed, the inclusion $\chi_t(\Fu(f))\subset \Fu(e^{it H}f)$ is just a restatement of Theorem \ref{teo7.4}. The opposite inclusion is equivalent to $\chi_t^{-1}\Fu(e^{i t H}f)\subset \Fu(f)$, namely to $\chi_t^{-1}\Fu(g)\subset \Fu(e^{-i t H}g)$, which is true by reversing the time.

\end{document}